\pdfoutput=1
\documentclass{article}

\usepackage{graphicx}
\usepackage{amsmath,amsfonts,amssymb,amsthm}
\newtheorem{proposition}{Proposition}
\newtheorem{corollary}{Corollary}
\theoremstyle{remark}
\newtheorem{remark}{Remark}
\newtheorem{example}{Example}
\begin{document}

\title{High order weak approximation schemes for L\'evy-driven SDEs}
\author{Peter Tankov\footnote{Peter Tankov, Centre de Math\'ematiques Appliqu\'ees,
  Ecole Polytechnique, Palaiseau, France. Email \texttt{peter.tankov@polytechnique.org}}
}
\date{}
\maketitle

\begin{abstract}
We propose new jump-adapted weak approximation schemes for stochastic
differential equations driven by pure-jump L\'evy processes. The idea
is to replace the driving L\'evy process $Z$ with a finite
intensity process which has the same L\'evy measure outside a
neighborhood of zero and matches a given number of moments of
$Z$. By matching 3 moments we construct a scheme which works
for all L\'evy measures and is superior to the existing approaches both
in terms of convergence rates and easiness of implementation. In the
case of L\'evy processes with stable-like behavior of small jumps, we
construct schemes with arbitrarily high rates of convergence by
matching a sufficiently large number of moments. 
\end{abstract}

\medskip

\noindent Key words: L\'evy-driven stochastic differential equation, Euler
scheme, high order discretization schemes, jump-adapted discretization, weak approximation.

\medskip

\noindent 2010 Mathematics Subject Classification: Primary 60H35, Secondary 65C05, 60G51. 

\section{Introduction}
\label{pt.intro}

Let $Z$ be a $d$-dimensional L\'evy process without diffusion component, that
is,
\[
Z_{t}=\gamma t + \int_{0}^{t}\int_{|y|\leq1}y\widehat{N}(dy,ds)+\int_{0}%
^{t}\int_{|y|>1}yN(dy,ds),\quad t\in\lbrack0,1].
\]
Here $\gamma\in\mathbb{R}^{d}$, $N$ is a Poisson random measure on
$\mathbb{R}^{d} \times[0,\infty)$ with intensity $\nu$ satisfying $\int
1\wedge\|y\|^{2}\nu(dy)<\infty$ and $\widehat{N}(dy,ds)=N(dy,ds)-\nu(dy)ds$
denotes the compensated version of $N$. We study the
case when $\nu(\mathbb{R}^{d})=\infty$, that is, there is an infinite number of jumps in every interval of nonzero length a.s. Further, let $X$ be an $\mathbb{R}^{n}$-valued
adapted stochastic process, unique solution of the stochastic differential
equation
\begin{align}
X_{t}=X_{0}+\int_{0}^{t}h(X_{s-})dZ_{s},\quad t\in\lbrack0,1], \label{pt.sde}%
\end{align}
where $h$ is an $m\times d$ matrix.

In this article we are interested in the numerical evaluation of
$E[f(X_1)]$ for a sufficiently smooth function $f$ by Monte Carlo, via discretization
and simulation of the process $X$. We propose new weak approximation
algorithms for \eqref{pt.sde} and study their rate of convergence. 

The traditional method to simulate
$X$ is to use the Euler scheme with constant time step
$$
\hat X^n_{\frac{i+1}{n}} = \hat X^n_{\frac{i}{n}} + h(\hat X^n_{\frac{i}{n}})(Z_{\frac{i+1}{n}}-Z_{\frac{i}{n}}).
$$
This method has the convergence rate \cite{protter_talay,jkmp.05} 
$$
|E[f(X_1)]-E[f(\hat X^n_1)]|\leq \frac{C}{n}
$$ 
but suffers from
two difficulties: first, for a general L\'{e}vy measure $\nu$, there is no
available algorithm to simulate the increments of the driving L\'{e}vy process and second, a large jump of $Z$ occurring between two discretization points can
lead to an important discretization error.

A natural idea due to Rubenthaler \cite{rubenthaler} (in the context of
finite-intensity jump processes, this idea appears also in
\cite{bruti_liberati.platen.07,mstz.08}), 
is to approximate $Z$ with a compound Poisson process by replacing the
small jumps with their expectation
$$
Z^\varepsilon_t:= \gamma_\varepsilon t + \int_0^t\int_{|y|>\varepsilon}yN(dy,ds),\quad \gamma_\varepsilon = \gamma - \int_{\varepsilon<|y|\leq 1}y\nu(dy),
$$
and then place discretization dates at all jump times of
$Z^\varepsilon$.

The computational complexity of simulating a single trajectory using
this method becomes a random variable, but the convergence rate may be
computed in terms of the \emph{expected} number of discretization dates, proportional to $\lambda_\varepsilon = \int_{|y|\geq \varepsilon}\nu(dy)$. When the jumps of $Z$ are
highly concentrated around zero, however, this approximation is too
rough and the convergence rates can be arbitrarily slow.

In \cite{kohatsu.tankov.09}, the authors proposed a scheme which
builds on Rubenthaler's idea of using the times of large jumps of $Z$ as
discretization dates but achieves better convergence rates. Their idea
is, first, to approximate the small jumps of $Z$ with a suitably
chosen Brownian motion, in order to match not only the first but also
the second moment of $Z$, and second, to construct an approximation to
the solution of the  continuous SDE between the times of large
jumps. 
Similar ideas of Gaussian correction were recently used in
\cite{dereich.11} in the context of multilevel Monte Carlo methods for
the problem \eqref{pt.sde}. 
However, although diffusion approximation of small jumps
improves the convergence rate, there are limits on how well the small
jumps of a L\'evy process can be approximated by a Brownian motion. In
particular, the Brownian motion is a symmetric process, while a L\'evy
process may be asymmetric.


In this paper we develop new jump-adapted discretization schemes based
on approximating the L\'evy process $Z$ with a finite intensity L\'evy
process $Z^\varepsilon$ without diffusion part. Contrary to previous
works, instead of simply truncating jumps smaller than $\varepsilon$,
we construct efficient finite intensity approximations which match a
given number of moments of $Z$. These approximations are superior to
the existing approaches in two ways. First, given that $Z^\varepsilon$
is a finite intensity L\'evy process, the solution to \eqref{pt.sde} with
$Z$ replaced by $Z^\varepsilon$ is easy to compute, either explicitly
or with a fast numerical method, making it straightforward to
implement the scheme. Second, by choosing the parameters of
$Z^\varepsilon$ in a suitable manner, one can, in principle, match an
arbitrary number of moments of $Z$ and obtain a discretization scheme
with an arbitrarily high convergence rate.

The paper is structured as follows. In Section \ref{pt.mmatch}, we present
the main idea of moment matching approximations and provide a basic
error bound for such schemes. In Section \ref{pt.3mom}, we introduce our
first scheme which is based on matching 3 moments of $Z$ and can be used for general L\'evy
  processes. For L\'evy processes with stable-like behavior of small
  jumps near zero, the scheme is shown to be rate-optimal. 
Finally, Section \ref{pt.highorder} shows how schemes of arbitrary
  order can be constructed by matching additional moments, once again,
  in the context of L\'evy processes with stable-like behavior of small jumps.

\section{Moment matching compound Poisson approximations}
\label{pt.mmatch}
Let $Z^\varepsilon$ be a finite intensity L\'evy process without
diffusion part  approximating $Z$ in a certain sense to be defined later: 
\begin{align}
Z^\varepsilon_t:= \gamma_\varepsilon t + \int_0^t\int_{\mathbb
  R^d}yN^\varepsilon(dy,ds),
\label{pt.Zeps}
\end{align}
where $N^\varepsilon$ is a Poisson random measure with intensity
measure $dt\times \nu^\varepsilon$ such that $\lambda_\varepsilon :=
\nu^\varepsilon(\mathbb R^d)<\infty$. 

In this paper we propose to approximate the process \eqref{pt.sde} by the
solution to
\begin{align}
d\hat X_t = h(\hat
  X_{t-})dZ^\varepsilon_t,\quad \hat X_0 = X_0,\label{pt.approxsol}
\end{align}
which can be computed by applying the Euler scheme at the jump times
of $Z^\varepsilon$ and solving the deterministic ODE $d\hat X_t =
h(\hat X_t) \gamma_\varepsilon dt$ explicitly (or by a Runge-Kutta
method\footnote{In this paper, to simplify the treatment, we assume
  that the ODE is solved explicitly. Upper bounds on the additional
  error introduced by the Runge-Kutta method are given in
  \cite[Proposition 7]{kohatsu.tankov.09}. These bounds can be made
  arbitrarily small by taking a Runge-Kutta algorithm of sufficiently high order. })
  between these jump times.  The following proposition provides a basic
  estimate for the weak error of such an approximation scheme. 
We impose the following alternative regularity assumptions
on the functions $f$ and $h$:
\begin{description}
\item[\textbf{$\mathbf{(H_n)}$}] $f\in C^n$, $h\in C^n$ $f^{(k)}$ and $h^{(k)}$ are bounded for $1\leq k \leq n$ and $\int z^{2n} \nu(dz)<\infty$.
\item[\textbf{$\mathbf{(H'_n)}$}] $f\in C^n$, $h\in C^n$, $h^{(k)}$ are bounded for $1\leq k \leq n$, $f^{(k)}$ have at most polynomial growth for $1\leq k \leq n$ and $\int |z|^{k} \nu(dz)<\infty$ for all $k\geq 1$. 
\end{description}
\begin{proposition}\label{pt.basicbound}
Let $Z$ and $\hat Z$ be L\'evy processes with characteristic triplets
$(0,\nu,\gamma)$ and $(0,\hat \nu,\hat \gamma)$ respectively, and let
$X$ and $\hat X$ be the corresponding solutions of SDE \eqref{pt.sde}. 
Assume $\hat \gamma = \gamma$, $\hat \nu = \nu$ on $\{\| x \| >1\}$, either $\mathbf{(H_{n})}$ or $\mathbf{(H^{\prime}_{n})}$ for
$n\geq 3$ and
\begin{align}
\int_{\mathbb R^d} x_{i_1} \dots x_{i_k} \nu(dx)  = \int_{\mathbb R^d}
x_{i_1} \dots x_{i_k} \hat \nu(dx), \quad 2 \leq k \leq n-1,\quad
1\leq i_k \leq d.\label{pt.momentmatch}
\end{align}
Then
\[
|E[f(\hat X_{1}) - f(X_{1})]| \leq C \int_{\mathbb R^d} \|x\|^n|d\nu -d \hat \nu|,
\]
where the constant $C$ may depend on $f$, $g$, $x$ and $\nu$ but not on
$\hat \nu$. 
\end{proposition}
\begin{proof} To simplify notation, we give the proof in the case
  $m=d=1$.
Let $u(t,x) = E^{(t,x)} [f(X_1)]$. By Lemma 13 in
\cite{kohatsu.tankov.09}, $u\in C^{1,n}([0,1]\times \mathbb R)$ and
satisfies
\begin{align}
\hspace*{-2cm}\frac{\partial u}{\partial t}(t,x)  &  + \gamma\frac{\partial
u}{\partial x}(t,x)h(x) +\int_{|y|>1}\left(  u(t,x+h(x)y)-u(t,x)\right)  \nu\left(  dy\right)
 \nonumber \\ &+ \int_{|y|\leq1}\left(  u(t,x+h(x)y)-u(t,x)-\frac
{\partial u}{\partial x}(t,x)h(x)y\right)  \nu(dy) =0, \label{pt.u_eq}\\
u(1,x)  &  =f(x).\nonumber
\end{align}
Applying It\^o formula under the integral sign and using \eqref{pt.u_eq}
and Lemma 11 in \cite{kohatsu.tankov.09} (bounds on moments of $\hat
X_t$) yields 
\begin{align*}
&E[f(\hat X_1) - f(X_1)]  = E[u(1,\hat X_1) - u(0,X_0)] \\
& = E\left[\int_0^1 \int_{\mathbb R} \left\{u(t,\hat X_t + h(\hat X_t) z) -
  u(t,\hat X_t) - h(\hat X_t) z \frac{\partial u}{\partial
    x}\right\}(d\nu_\varepsilon - d\nu) dt \right] \\
&+ E\left[\int_0^1 \int_{\mathbb R} \left\{u(t,\hat X_{t-} + h(\hat X_{t-}) z) -
  u(t,\hat X_{t-})\right\}\hat N(dz, dt) \right]\\
&= E\left[\int_0^1 \int_{\mathbb R} \sum_{k=2}^{n-1} \frac{\partial^k
  u (t,\hat X_t)}{\partial x^k} h^k(\hat X_t) z^k (d\nu_\varepsilon -
d\nu) dt + \text{remainder}\right],\\
& = E[\text{remainder}],
\end{align*}
where in the last line we used the moment matching condition
\eqref{pt.momentmatch} and the remainder coming from the Taylor formula
can be estimated as
\begin{align*}
|\text{remainder}| & \leq \int_0^1 \int_{\mathbb R} \sup_{0\leq s \leq
1} \left| \frac{\partial^n
  u (s,\hat X_s)}{\partial x^n}\right|  |h(\hat X_t)|^n |z|^n |d\nu_\varepsilon -
d\nu| dt\\
& \leq C  \sup_{0\leq s \leq
1} \left| \frac{\partial^n
  u (s,\hat X_s)}{\partial x^n}\right| \sup_{0\leq s\leq 1} |h(\hat X_s)|^n \int_{\mathbb R} |z|^n |d\nu_\varepsilon -
d\nu| 
\end{align*}
From the Lipschitz property of $h$ and Lemma 13 in \cite{kohatsu.tankov.09},
$$
\sup_{0\leq s \leq
1} \left| \frac{\partial^n
  u (s,\hat X_s)}{\partial x^n}\right| \sup_{0\leq s\leq 1} |h(\hat
X_s)|^n \leq C(1+ \sup_{0\leq t \leq 1}|\hat X_t|^p)
$$
for some $C<\infty$, where $p=n$ under $\mathbf{(H_{n})}$ and $p>n$
under $\mathbf{(H'_{n})}$. Following the arguments in the
proof of Lemma 11 in \cite{kohatsu.tankov.09}, we get
$$
E[\sup_{0\leq t \leq 1} |\hat X_t|^p] \leq
C(1+|x|^p)\exp\left[c\left(|\bar \gamma|^p +\int_{\mathbb R} |z|^p
    \hat\nu(dz) + \left(\int_{\mathbb R}z^2 \hat\nu(dz)\right)^{p/2}\right)\right]
$$
for different constants $C$ and $c$, where 
$$
\bar \gamma = \hat \gamma + \int_{|z|>1} z \hat \nu(dz) = \gamma + \int_{|z|>1} z \nu(dz)
$$
by our assumptions. Since $\int_{\mathbb R}z^2 \hat\nu(dz) =
\int_{\mathbb R}z^2 \nu(dz)$ by assumption, and 
$$\int_{\mathbb R} |z|^p
    \hat\nu(dz) \leq \int_{|z|>1} |z|^p
    \hat\nu(dz) + \int_{|z|\leq 1} |z|^2 \hat\nu(dz) = \int_{|z|>1} |z|^p
    \nu(dz) + \int_{|z|\leq 1} |z|^2 \nu(dz),
$$
it is clear that $E[\sup_{0\leq t \leq 1} |\hat X_t|^p]\leq C$ for
some constant $C$ which does not depend on $\hat \nu$. 

\end{proof}

\section{The 3-moment scheme}
\label{pt.3mom}
Our first scheme is based on matching the first 3 moments of the
process $Z$. Let $S^{d-1}$ be the unit sphere in the $d$-dimensional
space, and $\nu(dr\times d\theta)$ be a L\'evy measure on $\mathbb R^d$
written in
spherical coordinates $r\in [0,\infty)$ and $\theta \in S^{d-1}$ and
satisfying $\int_{[0,\infty) \times S^{d-1}} r^3
\nu(dr,d\theta)<\infty$.  Denote by
$\bar \nu$ the reflection of $\nu$ with respect to the origin defined by $\bar\nu(B)
= \nu(\{x:-x \in B\})$.  We
introduce two measures on $S^{d-1}$:
\begin{align*}
\bar \lambda (d\theta) &= \frac{1}{2}\int_{|r|\leq \varepsilon}
\frac{r^3}{\varepsilon^3} \left(\nu(dr,d\theta) - \bar \nu(dr,d\theta)\right) \\
\lambda (d\theta) &= \frac{1}{2}\int_{|r|\leq \varepsilon} \frac{r^2}{\varepsilon^2} \left(\nu(dr,d\theta) + \bar\nu(dr,d\theta) \right).
\end{align*}
The \emph{3-moment scheme} is defined by 
\begin{align}
\nu_\varepsilon(dr, d\theta) &=  \nu(dr,d\theta) 1_{r>\varepsilon} +
\delta_\varepsilon(dr) \{\lambda(d\theta) + \bar \lambda (d\theta)\}\label{pt.3mom1}\\
\gamma_\varepsilon & = \gamma - \int_{[0,1]\times S^{d-1}} r\theta\,
\nu_\varepsilon(dr,d\theta), \label{pt.3mom2}
\end{align}
where $\delta_\varepsilon$ denotes a point mass at $\varepsilon$. 

\begin{proposition}[Multidimensional 3-moment scheme]\label{pt.3mom.prop}
For every $\varepsilon>0$, $\nu_\varepsilon$ is a finite positive measure satisfying
\begin{align}
\int_{\mathbb R^d} x_i x_j \nu(dx) &= \int_{\mathbb R^d} x_i x_j \nu_\varepsilon (dx) \label{pt.2ndmom}\\
 \int_{\mathbb R^d} x_i x_j x_k \nu(dx) &= \int_{\mathbb R^d} x_i x_j
 x_k \nu_\varepsilon (dx),\quad 1\leq i,j,k\leq d\\
\lambda_\varepsilon := \int_{\mathbb R^d} \nu_\varepsilon(dx) &= \int_{\|x\|>\varepsilon} \nu(dx) + \varepsilon^{-2}\int_{\|x\|\leq \varepsilon} \|x\|^2 \nu(dx)\\
\int_{\mathbb R^d} \|x\|^4 |d\nu - d\nu_\varepsilon| \leq  & \int_{\|x\|\leq \varepsilon} \|x\|^4 \nu(dx) + \varepsilon^2 \int_{\|x\|\leq \varepsilon} \|x\|^2 \nu(dx),
\end{align}
where the last inequality is an equality if $\nu(\{x: \|x\|=\varepsilon\}) = 0$.
\end{proposition}
\begin{proof}
The positivity of $\nu_\varepsilon$ being straightforward, let us
check \eqref{pt.2ndmom}. Let $\{e_i\}_{i=1}^d$ be the coordinate
vectors. Then,
\begin{align*}
&\int_{\mathbb R^d} x_i x_j \nu_\varepsilon (dx) =
\int_{[0,\infty)\times S^{d-1}} r^2 \langle \theta,e_i \rangle \langle
\theta, e_j \rangle \nu_\varepsilon (dr, d\theta) \\
& = \int_{(\varepsilon,\infty)\times S^{d-1}} r^2 \langle \theta,e_i \rangle \langle
\theta, e_j \rangle \nu (dr, d\theta)  + \int_{S^{d-1}} \varepsilon^2 \langle \theta,e_i \rangle \langle
\theta, e_j \rangle \{\lambda (d\theta) + \bar \lambda (d\theta)\} \\
& = \int_{(\varepsilon,\infty)\times S^{d-1}} r^2 \langle \theta,e_i \rangle \langle
\theta, e_j \rangle \nu (dr, d\theta)  + \int_{S^{d-1}} \varepsilon^2 \langle \theta,e_i \rangle \langle
\theta, e_j \rangle \lambda (d\theta) \\
& = \int_{(0,\infty)\times S^{d-1}} r^2 \langle \theta,e_i \rangle \langle
\theta, e_j \rangle \nu (dr, d\theta) = \int_{\mathbb R^d} x_i x_j \nu (dx).
\end{align*}
The other equations can be checked in a similar manner. 
\end{proof}
\begin{corollary}
Let $d=1$. Then the 3-moment scheme can be written as 
\begin{align*}
\nu_\varepsilon(dx) &=  \nu(dx) 1_{|x|>\varepsilon} + \lambda_+
\delta_\varepsilon(dx) + \lambda_- \delta_{-\varepsilon}(dx)\\
\lambda_\pm &= \frac{1}{2}\left\{\int_{|x|\leq \varepsilon} \frac{x^2}{\varepsilon^2} \nu(dx) \pm \int_{|x|\leq \varepsilon} \frac{x^3}{\varepsilon^3} \nu(dx) \right\}
\end{align*}
\end{corollary}
\begin{corollary}[Worst-case convergence rate]
Assume $\mathbf{(H_{4})}$ or $\mathbf{(H^{\prime}_{4})}$. Then the
solution $\hat X$ of \eqref{pt.approxsol} with the characteristics of
$Z^\varepsilon$ given by \eqref{pt.3mom1}--\eqref{pt.3mom2} satisfies 
$$
|E[f(\hat X_1) - f(X_1)]| = o(\lambda_\varepsilon^{-1}). 
$$
as $\varepsilon \to 0$.
\end{corollary}
\begin{proof}
By Proposition \ref{pt.basicbound} we need to show that 
$$
\lim_{\varepsilon\downarrow 0} \lambda_{\varepsilon} \int_{\mathbb
  R^d} \|x\|^4 |d\nu - d\nu_\varepsilon| = 0.
$$
By Proposition \ref{pt.3mom.prop}, 
\begin{align*}
&\lim_{\varepsilon\downarrow 0} \lambda_{\varepsilon} \int_{\mathbb
  R^d} \|x\|^4 |d\nu - d\nu_\varepsilon| \\
&\leq  \lim_{\varepsilon\downarrow 0}\left\{\int_{\|x\|>\varepsilon}
  \nu(dx) + \varepsilon^{-2}\int_{\|x\|\leq \varepsilon} \|x\|^2
  \nu(dx)\right\} \\ &\qquad \qquad \times
\left\{ \int_{\|x\|\leq \varepsilon} \|x\|^4 \nu(dx) + \varepsilon^2
  \int_{\|x\|\leq \varepsilon} \|x\|^2 \nu(dx)\right\} \\
&\leq 2 \lim_{\varepsilon\downarrow 0}\varepsilon^2\left\{\int_{\|x\|>\varepsilon} \nu(dx) + \varepsilon^{-2}\int_{\|x\|\leq \varepsilon} \|x\|^2 \nu(dx)\right\} 
  \int_{\|x\|\leq \varepsilon} \|x\|^2 \nu(dx) \\
 &= 2 \lim_{\varepsilon\downarrow 0}\varepsilon^2\int_{\|x\|>\varepsilon} \nu(dx)  
  \int_{\|x\|\leq \varepsilon} \|x\|^2 \nu(dx) \\ &\leq 2  \int_{\mathbb
    R^d} \|x\|^2 \nu(dx) \lim_{\varepsilon\downarrow
    0}\varepsilon^2\int_{\|x\|>\varepsilon} \nu(dx)  = 0,
\end{align*}
where in the last line the dominated convergence theorem was used. 
\end{proof}

In many parametric or semiparametric models, the L\'evy measure has a
singularity of type
$\frac{1}{|x|^{1+\alpha}}$ near zero. This is the case for
stable processes, tempered stable processes \cite{rosinski.04}, normal
inverse Gaussian process \cite{bns_nig}, CGMY
\cite{finestructure} and other models. Stable-like behavior of small
jumps is a standard assumption for the analysis of asymptotic behavior
of L\'evy processes in many contexts, and in our problem as well, this
property allows to obtain a more precise estimate of the convergence rate. We shall impose the
following assumption, which does not require the L\'evy measure to have
a density:
\begin{description}
\item[\textbf{$\mathbf{(H-\alpha)}$}] There exist $C>0$ and $\alpha
  \in (0,2)$ such that\footnote{Throughout this paper we write $f\sim g$ if $\lim\frac{f(x)}{g(x)} = 1$ and
    $f\lesssim g$ if $\limsup\frac{f(x)}{g(x)} \leq 1$. }
\begin{align}
l(r)\sim C r^{-\alpha}\quad \text{as} \quad r\to 0\label{pt.stable.eq}
\end{align} 
where $l(r):= \int_{\|x\|>r} \nu(dx)$. 
\end{description}
\begin{corollary}[Stable-like behavior]
Assume $\mathbf{(H-\alpha)}$ and $\mathbf{(H_{4})}$ or $\mathbf{(H^{\prime}_{4})}$. 
Then the
solution $\hat X$ of \eqref{pt.approxsol} with the characteristics of
$Z^\varepsilon$ given by \eqref{pt.3mom1}--\eqref{pt.3mom2} satisfies 
$$
|E[f(\hat X_1) - f(X_1)]| = O\left(\lambda_\varepsilon^{1-\frac{4}{\alpha}}\right). 
$$
\end{corollary}
\begin{proof}
Under $\mathbf{(H-\alpha)}$, by integration parts we get that for all
$n\geq 2$, 
$$
\int_{\|x\|\leq r} \|x\|^n \nu(dx) \sim
\frac{C\alpha}{n-\alpha} r^{n-\alpha}\quad \text{as $r\to 0$.}
$$
Therefore, under this assumption,
$$
\lambda_\varepsilon \sim \frac{2C}{2-\alpha}
\varepsilon^{-\alpha}\quad \text{and}\quad \int_{\mathbb R^d} \|x\|^4
|d\nu - d\nu_\varepsilon| \lesssim
\frac{C\alpha(6-2\alpha)}{(2-\alpha)(4-\alpha)}
\varepsilon^{4-\alpha}\quad \text{as $\varepsilon\to 0$,}
$$
from which the result follows directly. 
\end{proof}

\paragraph{Rate-optimality of the 3-moment scheme}
From Proposition \ref{pt.basicbound} we know that under the assumption $\mathbf{(H_4)}$ or
$\mathbf{(H'_4)}$, the approximation error of a scheme of the form
\eqref{pt.Zeps}--\eqref{pt.approxsol} can be measured in terms of the $4$-th absolute moment
of the difference of L\'evy measures.
We introduce the class of L\'evy
measures on $\mathbb R^d$ with intensity bounded by $N$:
$$
\mathcal M^N = \{\nu \ \text{L\'evy measure on $\mathbb R^d$,}\  \nu(\mathbb R^d)\leq N\}.
$$
The class of L\'evy measures with intensity bounded by $\lambda_\varepsilon$ is then denoted by  $\mathcal
M^{\lambda_\varepsilon}$, and the smallest possible error achieved by
any measure within this class is bounded from below by a constant
times $\inf_{\nu' \in
  \mathcal M^{\lambda_\varepsilon}} \int_{\mathbb R^d} \|x\|^4 |d\nu -
d\nu'|$. The next result shows that as $\varepsilon\to 0$, the error
achieved by the 3-moment scheme $\nu_\varepsilon$ differs from this
lower bound by at most a constant multiplicative factor .

\begin{proposition}
Assume $\mathbf{(H-\alpha)}$ and let $\nu_\varepsilon$ be given by \eqref{pt.3mom1}. Then,
$$
\limsup_{\varepsilon \downarrow 0} \frac{\int_{\mathbb R^d} \|x\|^4 |d\nu - d\nu_\varepsilon|}{\inf_{\nu' \in \mathcal M^{\lambda_\varepsilon}} \int_{\mathbb R^d} \|x\|^4 |d\nu - d\nu'|} < \infty. 
$$
\end{proposition}
\begin{proof}
\noindent \textit{Step 1.} Let us first compute
\begin{align}
\mathcal E_N := \inf_{\nu' \in \mathcal M^{N}} \int_{\mathbb R^d} \|x\|^4 |d\nu - d\nu'|.\label{pt.optnu}
\end{align}
For $\nu' \in \mathcal M^{N}$, let $\nu' = \nu'_c + \nu'_s$ where
$\nu'_c$ is absolutely continuous with respect to $\nu$ and $\nu'_s$
is singular. Then $\nu'_c(\mathbb R^d)\leq N$ and 
$$
\int_{\mathbb R^d} \|x\|^4 |d\nu - d\nu'| = \int_{\mathbb R^d} \|x\|^4
|d\nu - d\nu'_c| + \int_{\mathbb R^d} \|x\|^4 d\nu'_s.
$$
Therefore, the minimization in \eqref{pt.optnu} can be restricted to
measures $\nu'$ which are absolutely continuous with respect to $\nu$,
or, in other words,
\begin{align*}
\mathcal E_N = \inf \int_{\mathbb R^d} \|x\|^4 |1-\lambda(x)| \nu(dx),
\end{align*}
where the $\inf$ is taken over all measurable functions
$\lambda:\mathbb R^d \to \mathbb R^+$ such that $\int_{\mathbb R^d}
\lambda(x) \nu(dx)\leq N$. By a similar argument, one can show that
it is sufficient to consider only functions $\lambda:\mathbb R^d \to
[0,1]$. Given such a function $\lambda(r,\theta)$, the spherically
symmetric function 
$$
\hat \lambda(r):= \frac{\int_{S^{d-1}} \lambda(r,\theta)\nu(dr,d\theta)}{\int_{S^{d-1}} \nu(dr,d\theta)}
$$
leads to the same values of the intensity and the minimization
functional. Therefore, letting $\hat \nu(dr):=\int_{S^{d-1}}
\nu(dr,d\theta)$,
\begin{align}
\mathcal E_N = \inf_{0\leq \hat \lambda\leq 1} \int_{0}^\infty r^4 (1-\hat\lambda(r))
\hat\nu(dr)\quad \text{s.~t.}\quad \int_0^\infty
\hat\lambda(r)\hat \nu(dr)\leq N. \label{pt.lambdahat}
\end{align}
For every $e>0$,
$$
\mathcal E_N \geq \inf_{0\leq \hat \lambda  \leq 1} \left\{\int_{0}^\infty r^4 (1-\hat\lambda(r))
\hat\nu(dr) + e^4 \left(\int_0^\infty
\hat\lambda(r)\hat \nu(dr)- N\right)\right\}.
$$
The $\inf$ in the right-hand side can be computed pointwise and is
attained by $\hat \lambda(r) = 1_{r>e} + \mu
1_{r=e}$ for any $\mu\in[0,1]$. Let $e(N)$ and
$\mu(N)$ be such that 
$$\hat \nu((e(N),\infty)) + \mu(N)
\hat \nu(\{e(N)\}) = N.$$ 
Such a $e(N)$ can always be
determined uniquely and $\mu(N)$ is determined uniquely if
$\nu(\{e(N)\})>0$. It follows that $\hat \lambda(r) = 1_{r>e(N)} + \mu(N)
1_{r=e(N)}$ is a minimizer for \eqref{pt.lambdahat} and therefore
$$
\mathcal E_N = \int_{\|x\|<e} \|x\|^4 \nu(dx) +
(1-\mu) e^4 \nu(\{x:\|x\|=e\}),
$$
where $e$ and $\mu$ are solutions of 
$$
\nu(\{x:\|x\|>e\}) + \mu\nu(\{x:\|x\|=e\}) = N.
$$

\noindent \textit{Step 2.} For every $\varepsilon>0$, let
$e(\varepsilon)$ and $\mu(\varepsilon)$ be solutions of 
$$
\nu(\{x:\|x\|>e(\varepsilon)\}) + \mu(\varepsilon)\nu(\{x:\|x\|=e(\varepsilon)\}) = \lambda_\varepsilon.
$$
It is clear that $e(\varepsilon) \to 0$ as $\varepsilon \to 0$ and
after some straightforward computations using the assumption
$\mathbf{(H-\alpha)}$ we get that
$$
\lim_{\varepsilon \to 0} \frac{e(\varepsilon)}{\varepsilon} = \left(\frac{2-\alpha}{2}\right)^{1/\alpha}.
$$
Then,
\begin{align*}
&\lim_{\varepsilon \downarrow 0} \frac{\int_{\mathbb R^d} \|x\|^4
  |d\nu - d\nu_\varepsilon|}{\mathcal E_{\lambda_\varepsilon}}  = \lim_{\varepsilon \downarrow 0} \frac{\int_{\mathbb R^d} \|x\|^4
  |d\nu - d\nu_\varepsilon|}{\varepsilon^{4-\alpha}} \lim_{\varepsilon\downarrow 0}
\frac{\varepsilon^{4-\alpha}}{e(\varepsilon)^{4-\alpha}} \\
&\qquad\times\lim_{\varepsilon\downarrow 0}\frac{e(\varepsilon)^{4-\alpha}}{\int_{\|x\|<e(\varepsilon)} \|x\|^4 \nu(dx) +
(1-\mu(\varepsilon)) e(\varepsilon)^4 \nu(\{x:\|x\|=e(\varepsilon)\})}
\end{align*}
Under $\mathbf{(H-\alpha)}$ the three limits are easily computed and
we finally get
\begin{align}
\lim_{\varepsilon \downarrow 0} \frac{\int_{\mathbb R^d} \|x\|^4
  |d\nu - d\nu_\varepsilon|}{\mathcal E_{\lambda_\varepsilon}}  =
(3-\alpha) \left(\frac{2}{2-\alpha}\right)^{4/\alpha}. \label{pt.constant}
\end{align}

\end{proof}
\begin{remark}
The constant $(3-\alpha) \left(\frac{2}{2-\alpha}\right)^{4/\alpha}>1$
appearing in the right-hand side of \eqref{pt.constant} cannot be
interpreted as a ``measure of suboptimality'' of the 3-moment scheme,
but only as a rough upper bound, because in the optimization problem
\eqref{pt.optnu} the moment-matching constraints were not imposed (if
they were, it would not be possible to solve the problem
explicitly). On the other hand, the fact that this constant is
unbounded as $\alpha \to 2$ suggests that such a rate-optimality
result cannot be shown for general L\'evy measures without imposing the
assumption $\mathbf{(H-\alpha)}$.
\end{remark}

\paragraph{Numerical illustration}
We shall now illustrate the theoretical results on a concrete example
of a SDE driven by a normal inverse Gaussian (NIG) process
\cite{bns_nig}, whose characteristic function is
$$
\phi_t(u):=E[e^{iuZ_t}] = \exp\left\{-\delta t \left(\sqrt{\alpha^2 - (\beta-iu)^2} - \sqrt{\alpha^2 - \beta^2}\right)\right\},
$$
where $\alpha>0$, $\beta \in (-\alpha,\alpha)$ and $\delta>0$ are parameters. The L\'evy density is given by
$$
\nu(x) = \frac{\delta \alpha}{\pi} \frac{e^{\beta x} K_1(\alpha|x|)}{|x|},
$$
where $K$ is the modified Bessel function of the second kind. The NIG
process has stable-like behavior of small jumps with $\nu(x)\sim
\frac{const}{|x|^2}$, $x\to 0$ (which means that
$\mathbf{(H-\alpha)}$ is satisfied with $\alpha=1$), and exponential
tail decay.
The increments of the NIG process can be simulated explicitly (see \cite[algorithms 6.9 and 6.10]{levybook}), which enables us to compare our jump-adapted algorithm with the classical Euler scheme.

For the numerical example we solve the one-dimensional SDE 
$$
dX_t = \sin(a X_t) dZ_t,
$$ 
where $Z$ is the NIG L\'evy process (with drift adjusted to have $E[Z_t]=0$). 
The solution of the corresponding deterministic ODE 
$$
d X_t = \sin(a X_t) dt, \quad X_0 = x
$$
is given explicitly by
$$
X_t = \theta(t;x) = \frac{1}{a} \arccos \frac{1+\cos(a x) - e^{2a t}(1-\cos(ax))}{1+\cos(a x) + e^{2a t}(1-\cos(ax))}
$$

Figure \ref{pt.nig.fig} presents the approximation errors for evaluating
the functional $E[(X_1-1)^+]$ by Monte-Carlo using the 3-moment scheme
described in this section (marked with crosses), the diffusion
approximation of \cite{kohatsu.tankov.09} (circles) and the classical Euler
scheme (diamonds).  The parameter values are $\sigma=0.5$, $\theta=0.4$,
$\kappa=0.6$, $a=5$ and $X_0=1$. 
For each scheme we plot the logarithm of the approximation error as
function of the logarithm of the computational cost (time needed to
simulate $10^6$ trajectories). The curves are
obtained by varying the truncation parameter $\varepsilon$ for the two
jump-adapted schemes and by varying the discretization time step for
the Euler scheme.

The
approximation error for the Euler scheme is a straight line with slope
corresponding to the theoretical convergence rate of
$\frac{1}{n}$. The graph for the 3-moment scheme seems to confirm the
theoretical convergence rate of $\lambda_\varepsilon^{-3}$; the scheme
is much faster than the other two and the corresponding curve quickly
drops below the dotted line which symbolizes the level of the
statistical error. 

\begin{figure}
\centerline{\includegraphics[width=0.7\textwidth]{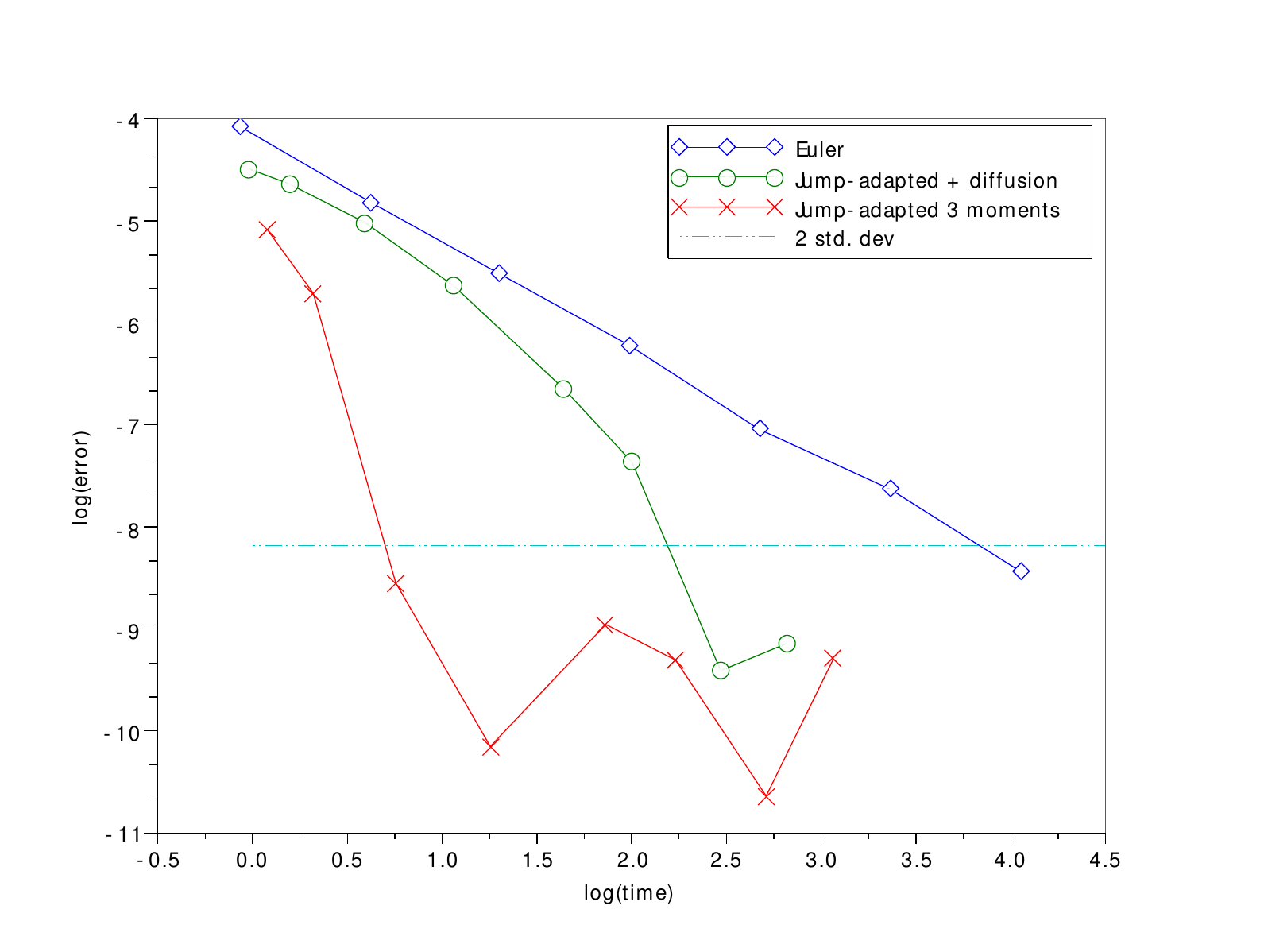}}
\caption{Approximation errors for the 3-moment scheme (crosses), the
  scheme of \cite{kohatsu.tankov.09} (circles) and the Euler scheme
  (diamonds). The horizontal dotted line corresponds to the logarithm
  of the two standard deviations of the Monte Carlo estimator (the
  standard deviation is about the same for the three schemes and
  independent of the discretization step); everything that is below
  the dotted line is Monte Carlo noise.  }
\label{pt.nig.fig}
\end{figure}

\section{High order schemes for stable-like L\'evy processes}
\label{pt.highorder}
In this section, we develop schemes of arbitrary order for L\'evy
processes with stable-like behavior of small jumps. Throughout this
section, we take $d=1$ and let $Z$ be a L\'evy process with
characteristic triplet $(0,\nu,\gamma)$ satisfying the
following refined version of $\mathbf{(H-\alpha)}$:
\begin{description}
\item[$\mathbf{(H'-\alpha)}$] There exist, $c_+\geq 0$, $c_-\geq 0$ with
  $c_+ + c_- >0$ and $\alpha \in (0,2)$ such that 
$$
\int_\varepsilon^\infty \nu(dx) \sim c_+ \varepsilon^{-\alpha}\quad
\text{and}\quad \int_{-\infty}^{-\varepsilon} \nu(dx) \sim c_-
\varepsilon^{-\alpha}\quad \text{as}\quad \varepsilon \downarrow 0
$$
\end{description}

Introduce the probability measure 
\begin{align}
\mu^*(x) := \frac{
  (2-\alpha) |x|^{1-\alpha} (c_+ 1_{0 \leq x\leq 1} + c_- 1_{-1 \leq x\leq 0})}{c_++c_-}.\label{pt.mustar}
\end{align}
Let $n\geq 0$ and $\varepsilon>0$. The high-order scheme for the stochastic differential equation
\eqref{pt.sde} based on $n+2$ moments and truncation level $\varepsilon$ is constructed as follows:
\begin{enumerate}
\item \label{pt.step1} Find a discrete probability measure 
$\bar \mu = \sum_{i=0}^n a^*_i \delta_{x_i}$ with 
\begin{align}
\int_\mathbb R x^k \bar \mu(dx) = \int_\mathbb R x^k  \mu^*(dx),\quad 1\leq k \leq n,\label{pt.momentmatch}
\end{align}
such that $x_0<x_1 < \dots < x_n$, $x_i \neq
  0$ for all $i$ and $a^*_i>0$ for all $i$.
\item Compute the coefficients $\{a^\varepsilon_i\}$ by solving the linear system
$$
\hspace*{-1cm}\sigma^2_\varepsilon \sum_{i=0}^n a^\varepsilon_i x_i^k \varepsilon^k
= \int_{|x|\leq \varepsilon} x^{2+k}\nu(dx),\ k=0,\dots, n,\
\sigma_\varepsilon^2 = \int_{|x|\leq \varepsilon} x^{2}\nu(dx).
$$
\item The high-order scheme is defined by 
\begin{align}
\nu_\varepsilon(dx) = \nu(dx)1_{|x|>\varepsilon} +
\sigma^2_\varepsilon\sum_{i=0}^n \frac{ a^\varepsilon_i
  \delta_{\varepsilon x_i} (dx)}{x_i^2 \varepsilon^2},\qquad
\gamma_\varepsilon = \gamma - \int_{|z|\leq 1}
z\nu_\varepsilon(dz).\label{pt.ho2}
\end{align}
\end{enumerate}
\begin{remark}
The first step in implementing the scheme is to solve the
moment-matching problem \eqref{pt.momentmatch} for measure $\mu^*$. The
existence of at least one solution to this problem with $x_0<x_1 <\dots x_n$ and
$a^*_i \geq 0$ for all $i$ is guaranteed by the classical
Caratheodory's theorem, but this problem admits, in general, an infinite number of solutions. Here we impose the additional condition $x_i \neq 0$ and
$a^*_i >0$ for all $i$, which should be checked on a case by case
basis in concrete realizations of the scheme (see Example \ref{pt.momentex}). 
\end{remark}
\begin{remark}
It is easy to see that the measure 
$$
\nu^*_\varepsilon(dx):=(\sigma^*_\varepsilon)^2\sum_{i=0}^n \frac{ a^*_i \delta_{\varepsilon
    x_i} (dx)}{x_i^2 \varepsilon^2},\quad  (\sigma^*_\varepsilon)^2:= \int_{|x|\leq \varepsilon}x^2 \nu^*(z)dz.
$$
matches the moments of orders $2,\dots,n+2$ of $\nu^*(x) 1_{|x|\leq
  \varepsilon}$, where $\nu^*$ is the measure given by
$$
\nu^*(x)=\frac{\alpha c_+1_{x>0} + \alpha c_-
  1_{x<0}}{|x|^{1+\alpha}},\quad
$$
that is, $\nu^*$ satisfies the assumption $\mathbf{(H'-\alpha)}$ with
equalities instead of equivalences. The idea of the method is to
replace the coefficients $\{a^*_i\}$ with a different set of
coefficients while keeping the same points $\{x_i\}$ to obtain a
measure which matches the moments of $\nu(x) 1_{|x|\leq
  \varepsilon}$. Therefore, the points $\{x_i\}$ do not depend on
the truncation parameter $\varepsilon$ while the coefficients
$\{a^\varepsilon_i\}$ depend on it. 
\end{remark}

\begin{example}\label{pt.momentex}
As an example we provide a possible solution of the moment matching
problem for $n=3$, which leads to
a 5-moment scheme (matching 3 moments of $\mu^*$ or $5$ moments of the
L\'evy process). We assume that $\mu^*$ has mass both on the positive
and the negative half-line: $c_+c_- >0$.   

The moments of $\mu^*$ are given by 
$$
m_k = \frac{2-\alpha}{k+2-\alpha}(\rho + (-1)^k (1-\rho)),\quad \rho:=
\frac{c_+}{c_++c_-}. 
$$
It is then convenient to look for the discrete measure matching the
first 3 moments of $\mu^*$ in the form
\begin{align}
\bar \mu = (1-\rho) (p\delta_{-\varepsilon_2} +
(1-p)\delta_{-\varepsilon_1})+\rho ((1-p)\delta_{\varepsilon_1} +
p\delta_{\varepsilon_2} ),\label{pt.mubar3}
\end{align}
where $p\in (0,1)$, $0<\varepsilon_1<\varepsilon_2$ are
parameters to be identified from the moment conditions
$$
(1-p)\varepsilon_1^k + p\varepsilon_2^k = \frac{2-\alpha}{k+2-\alpha},
\quad k=1,2,3.
$$
For the purpose of solving this system of equations, let $\mathcal E$
be a random variable such that $P[\mathcal E = \varepsilon_2] = p = 1-
P[\mathcal E = \varepsilon_1]$. From the moment conditions, we get:
\begin{align}
\bar \varepsilon&:= E[\mathcal E] = \frac{2-\alpha}{3-\alpha}\label{pt.epsbar3},\qquad
\sigma^2 := \text{Var}\, \mathcal E =
\frac{(2-\alpha)}{(4-\alpha)(3-\alpha)^2},\\
s &:= \frac{E[(\mathcal E - E[\mathcal E])^3]}{\sigma^3} =
2\frac{\alpha-1}{5-\alpha}\sqrt{\frac{4-\alpha}{2-\alpha}}.  \label{pt.skew3}
\end{align}
On the other hand, the skewness $s$ can be directly linked to the
weight $p$:
\begin{align}
s = \frac{1-2p}{\sqrt{p(1-p)}}\quad \Rightarrow \quad p = \frac{1}{2}
- \frac{1}{2}\text{sign}\, (s)\sqrt{\frac{s^2}{s^2 + 4}},\label{pt.prob3}
\end{align}
and the parameters $\varepsilon_1$ and $\varepsilon_2$ can be linked
to $\bar \varepsilon, \sigma$ and $p$:
\begin{align}
\varepsilon_{1} = \bar \varepsilon - \sigma \sqrt{\frac{p}{1-p}},\quad
\varepsilon_{2} = \bar \varepsilon  + \sigma \sqrt{\frac{1-p}{p}}.\label{pt.eps123}
\end{align}
The dependence of $\varepsilon_1$, $\varepsilon_2$
and $p$ on $\alpha$ is shown in Figure \ref{pt.moment.fig}: it is clear
from the graph that the constraints $p\in (0,1)$ and $0<\varepsilon_1
< \varepsilon_2$ are satisfied for all $\alpha \in (0,2)$: therefore, equations (\ref{pt.mubar3}--\ref{pt.eps123}) define a $4$-atom probability measure which matches the first 3
moments of $\mu^*$.
\begin{figure}

\centerline{\includegraphics[width=0.7\textwidth]{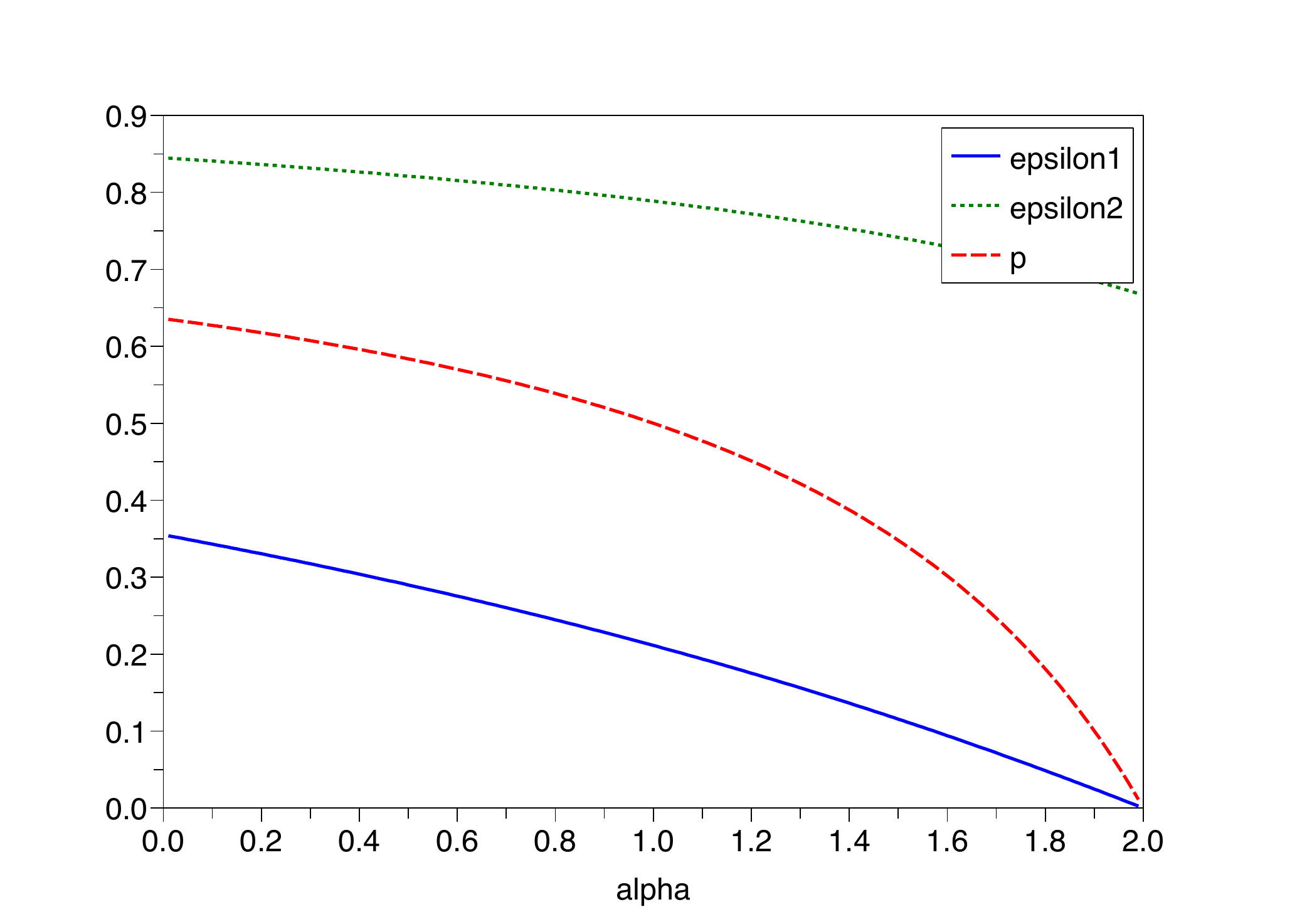}}
\caption{Solution of the moment matching problem for 3 moments (see
  Example \ref{pt.momentex}).}
\label{pt.moment.fig}
\end{figure}
\end{example}

\begin{proposition} Let $\{x_i\}$ be fixed according to \eqref{pt.momentmatch}. 
There exists $\varepsilon_0 >0$ such that for all
$\varepsilon<\varepsilon_0$, $\nu_\varepsilon$ is a positive measure satisfying
\begin{align}
&\int_{\mathbb R} x^k \nu(dx) = \int_{\mathbb R} x^k
\nu_{\varepsilon}(dx),\quad 2\leq k \leq n+2 \label{pt.ho.momc}
\end{align}
There exist positive constants $C_1$ and $C_2$ such that
$$
\lambda_\varepsilon = \nu_\varepsilon(\mathbb R)\sim C_1
\varepsilon^{-\alpha},\qquad \int_{\mathbb R} |x|^{n+3} |d\nu
  - d\nu_\varepsilon| \lesssim C_2 \varepsilon^{n+3-\alpha}\quad
  \text{as $\varepsilon\to 0$}. 
$$
\end{proposition}\label{pt.ho.prop}
\begin{corollary}Assume $(\mathbf{H_{n+3}})$ or
  $\mathbf{(H'_{n+3})}$. Then the solution $\hat X$ of \eqref{pt.approxsol}
  with characteristics of $Z^\varepsilon$ given by
  \eqref{pt.ho2} satisfies
$$
|E[f(\hat X_1) - f(X_1)]| = O\left(\lambda_\varepsilon^{1-\frac{n+3}{\alpha}}\right). 
$$
\end{corollary}

\begin{proof}[of Proposition \ref{pt.ho.prop}]
The moment conditions \eqref{pt.ho.momc} hold by construction.
Using integration by parts, we compute
\begin{align*}
\int_{|z|\leq \varepsilon} z^{2+k}\nu(dz) \sim \frac{(c_+ + (-1)^k
  c_-) \alpha \varepsilon^{2+k-\alpha}}{2+k-\alpha} \quad \text{as
  $\varepsilon\to 0$ for $k\geq 0$.}
\end{align*}
Therefore,
$$
\lim_{\varepsilon\to 0} \frac{1}{\sigma^2_\varepsilon \varepsilon^k}
\int_{|z|\leq \varepsilon} z^{2+k}\nu(dz) = \frac{(2-\alpha)(c_+ +
  (-1)^k c_-)}{(2+k-\alpha)(c_+ + c_-)} = \int_{\mathbb R }x^k \mu^*(dx).
$$
Since the matrix $M_{ij} = (x_j)^i$, $0\leq i \leq n$, $0\leq j \leq n$
is invertible (Vandermonde matrix), this implies that
$\lim_{\varepsilon \to 0} a^\varepsilon_i = a^*_i$. Therefore, there
exists $\varepsilon_0>0$ such that for all $\varepsilon<
\varepsilon_0$, $a^\varepsilon_i >0$ for all $i$ and $\nu_\varepsilon$
is a positive measure. 

We next compute:
\begin{align*}
\nu_\varepsilon(\mathbb R) &= \int_{|x|>\varepsilon}\nu(dx) +
\frac{\sigma_\varepsilon^2}{\varepsilon^2 }\sum_{i=0}^n
\frac{a_i^\varepsilon}{x_i^2} \sim \int_{|x|>\varepsilon}\nu(dx) +
\frac{\sigma_\varepsilon^2}{\varepsilon^2 }\sum_{i=0}^n
\frac{a_i^*}{x_i^2} \\
&\sim \varepsilon^{-\alpha} (c_+ + c_-) \left\{ 1 +
\frac{\alpha}{2-\alpha} \sum_{i=0}^n
\frac{a_i^*}{x_i^2}\right\},
\end{align*}
\begin{align*}
\int_{\mathbb R} |x|^{n+3} |d\nu
  - d\nu_\varepsilon| & \leq \int_{|x|\leq \varepsilon} |x|^{n+3} d\nu
  + \sigma_\varepsilon^2 \varepsilon^{n+1} \sum_{i=0}^n
{a_i^\varepsilon}{|x_i|^{n+1}} \\
&\sim \varepsilon^{n+3 -\alpha} (c_+ + c_-)
\left\{\frac{\alpha}{3+k-\alpha} + \frac{\alpha}{2-\alpha} \sum_{i=0}^n
{a_i^*}{|x_i|^{n+1}}\right\}.
\end{align*}

\end{proof}

\section*{Acknowledgement}
This research is supported by the Chair Financial Risks of the Risk Foundation sponsored by Soci\'et\'e G\'en\'erale, the Chair Derivatives of the Future sponsored by the F\'ed\'eration Bancaire
Fran\c caise, and the Chair Finance and Sustainable Development sponsored by EDF and Calyon.

\end{document}